\documentclass[11pt]{amsart}
\usepackage{amscd,amssymb,graphicx,enumerate,amssymb,setspace,amsfonts,amsthm,amscd,amsthm,amsmath,amssymb}
\usepackage[matrix,arrow,curve]{xy}

\newtheorem{theorem}{Theorem}[section]
\newtheorem{lemma}[theorem]{Lemma}
\newtheorem{prop}[theorem]{Proposition}
\newtheorem{proposition}[theorem]{Proposition}
\newtheorem{corollary}[theorem]{Corollary}
\newtheorem{cor}[theorem]{Corollary}
\newtheorem{conjecture}[theorem]{{Conjecture}}
\newtheorem{example}[theorem]{{Example}}
\newtheorem{problem}[theorem]{{Problem}}

\newtheorem{definition}[theorem]{{Definition}}
\newtheorem{claim}[theorem]{{Claim}}

\theoremstyle{remark}
\newtheorem{remark}[theorem]{Remark}

\def\bclaim{\begin{claim}}
\def\eclaim{\end{claim}}
\def\bdefin{\begin{definition}}
\def\edefin{\end{definition}}
\def\bcor{\begin{corollary}}
\def\ecor{\end{corollary}}
\def\bthm{\begin{theorem}}
\def\ethm{\end{theorem}}
\def\bconj{\begin{conjecture}}
\def\econj{\end{conjecture}}
\def\blem{\begin{lemma}}
\def\elem{\end{lemma}}
\def\blemma{\begin{lemma}}
\def\elemma{\end{lemma}}
\def\bprop{\begin{prop}}
\def\eprop{\end{prop}}
\def\bremark{\begin{remark}}
\def\eremark{\end{remark}}
\def\bhyp{\begin{hypothesis}}
\def\ehyp{\end{hypothesis}}
\def\bnot{\begin{notation}}
\def\enot{\end{notation}}
\def\bexample{\begin{example}}
\def\eexample{\end{example}}
\def\lb{\label}

 \def\be{\beta}

\def\K{K\"ahler } 
\def\KE{K\"ahler--Einstein } 
\def\KEE{K\"ahler--Einstein edge }

\def\h#1{\hbox{#1}}

\def\strutdepth{\dp\strutbox}
\def\specialstar{\vtop to \strutdepth{
    \baselineskip\strutdepth
    \vss\llap{$\star$\ \ \ \ \ \ \ \ \  }\null}}
\def\marginalstar{\strut\vadjust{\kern-\strutdepth\specialstar}}

\def\marginal#1{\strut\vadjust{\kern-\strutdepth
    {\vtop to \strutdepth{
    \baselineskip\strutdepth
    \vss\llap{{ \small #1 }}\null}
    }}
    }

\def\text{\textstyle}

\def\q{\quad} \def\qq{\qquad}

\def\ra{\rightarrow}

\def\sm{\setminus}

\newcommand{\PP}{{\mathbb P}} \newcommand{\RR}{\mathbb{R}}
 
 \newcommand{\NN}{{\mathbb N}}
\newcommand{\FF}{{\mathbb F}}

\def\beq{\begin{equation}}
\def\eeq{\end{equation}}
\def\bpf{\begin{proof}}
\def\epf{\end{proof}}

\def\eaeq{\end{aligned}}
\def\baeq{\begin{aligned}}

\def\aldp{asymptotically log del Pezzo }

\def\bi{\bibitem}

\makeatletter\@addtoreset{equation}{section} \makeatother

\date{}

\title{On flops and canonical metrics}

\author{Ivan A. Cheltsov and Yanir A. Rubinstein}

\address{\begin{tabbing}
\hspace*{28 em}\=\kill
University of Edinburgh 
\texttt{I.Cheltsov@ed.ac.uk}
\\
\\
University of Maryland 
\texttt{yanir@umd.edu}
\end{tabbing}}

\begin{document}

\maketitle

\begin{abstract}
This article is concerned with an observation for proving
non-existence of canonical K\"ahler metrics. The idea is to use a
rather explicit type of degeneration that applies in many
situations. Namely, in a variation on a theme introduced by
Ross--Thomas, we consider flops of the deformation to the normal
cone. This yields a rather widely applicable notion of stability
that is still completely explicit and readily computable, but with
wider scope. We describe some applications, among them, a proof
of one direction of the Calabi conjecture for
asymptotically logarithmic del Pezzo surfaces.

\end{abstract}

\pagestyle{headings}

\section{Motivation and results}
\label{section:intro}

A variety is slope stable in the sense of Ross--Thomas 
if, roughly, it is K-stable with respect to
degenerations to the normal cone of its subvarieties.
This notion has been studied extensively by a number of authors
and has yielded many non-existence results for canonical metrics
on projective \K manifolds. Our main purpose in this article
is to introduce a slight variation on this theme by considering
a somewhat more involved notion of stability that involves
additional flops on the degeneration to the normal cone but that
is still geometric and computable, and is partly inspired 
by the work of Li--Xu.
This gives many new non-existence results, and most notably
allows us to resolve one direction of the Calabi conjecture for
asymptotically logarithmic del Pezzo surfaces.

\subsection{Existence theorem for KEE metrics}

\KEE (KEE) metrics are a natural generalization of \KE metrics: they are smooth metrics on the complement of a divisor, and have a conical singularity of angle $2\pi\be$
transverse to that `complex edge'. When $\be=1$, of course, this is just
an ordinary \KE metric, that extends
smoothly across the divisor.
One can think of the metric as being `bent` at an angle $2\pi\be$ along the divisor.
 In the case of Riemann surfaces, KEE metrics are just the familiar constant curvature metrics with isolated cone singularities, that have been studied since the late 19th century, e.g., by Picard \cite{Picard1893}.

A basic question, whose origins trace 
back to Tian's 1994 lectures in the setting of nonpositive
curvature \cite{Tian1994},
extended in Donaldson's 2009 lectures to the setting
of anticanonical divisors on Fano manifolds \cite{Don2009}, and further extended
in our previous work \cite{CR} (see also the survey \cite[\S8]{R14}), 
is the following:

\begin{problem}{\it
Under what analytic conditions on the triple
$(X,D,\beta)$ does a KEE metric exist on
the \K manifold $X$ bent at an angle $2\pi\be$ along the divisor $D\subset X$?
}
\end{problem}
This is partly motivated by Troyanov's solution in the Riemann surface case
\cite{Troyanov},
and was settled
by Jeffres--Mazzeo--Rubinstein (sufficient condition)
and Darvas--Rubinstein (sufficient and necessary conditions) 
in higher dimensions for smooth $D$ \cite{JMR,DR}.
These results give an analytic criterion characterizing 
existence, once the cohomological condition
\beq\label{cohomEq}
-K_X-(1-\be)D \h{\ \ is\ $\mu$ times an ample class, for some $\mu\in\RR$},
\eeq
is satisfied.
\bremark
The analytic condition of \cite[Theorem~9.1]{DR} is optimal
and in particular improves on that
of \cite[Theorem~2]{JMR} in the presence of automorphisms.
An alternative proof of the sufficient condition of \cite{JMR} 
was later also given by Guenancia--Paun \cite{GuenanciaP},
who treated the more general case of a simple normal crossing (snc) $D$,
based on work of Berman et al. \cite{BBEGZ}; we
refer to the survey \cite{R14} for a thorough discussion and many more references.
\eremark

\subsection{
Angle increasing to $2\pi$}

The existence theorem of \cite{JMR} coupled with Berman's
work \cite{Berman} showed that KEE metrics always exist for $(X,D,\be)$
when $X$ is a Fano manifold admitting a smooth anticanonical divisor $D$
and $\be$ is small \cite[Corollary 1]{JMR}.

Following these results,
considerable amount of work about KEE metrics in recent years has
concerned the behavior of such
 metrics when the cone angle \emph{increases towards $2\pi$},
 the two main issues
 being to show
that when $X$ is Fano admitting a smooth anticanonical divisor $D$, then:
\begin{itemize}
\item [(a)] $X$ admits KEE metrics with angle $2\pi\beta$ along a
smooth anticanonical divisor for all angles $\beta<1$ sufficiently
close to $1$ iff $X$ is K-semistable;

\item[(b)] the limit of these KEE metrics as $\be$ tends to 1 is a smooth KE metric iff
$X$ is K-stable.
\end{itemize}
Problems (a)-(b) attracted a good deal of work building on combined
efforts of many researchers in the past two
decades, culminating in a solution \cite{CDS,T12}.

\subsection{Angle decreasing to $0$}

In \cite{CR}, we initiated a systematic study of the behavior in
the other extreme when the \emph{cone angle $\beta$ goes to zero}. In
partial analogy with the previous paragraph, the program initiated
in \cite{CR} concerns:
\begin{itemize}
\item [(a)] Determining all triples $(X,D,\be)$ satisfying
\eqref{cohomEq} with \emph{sufficiently small $\beta$};
\item [(b)] Obtaining a condition
equivalent to existence of KEE metrics for such triples;
\item[(c)] Understanding the limit, when such exists, of these KEE metrics as $\be$
tends to zero.
\end{itemize}

This program is largely open. In \cite{CR} we established
(a) in dimension two under a technical assumption that the
pair $(X,D)$ is strongly asymptotically log Fano (see Definition 
\ref{definition:log-Fano};  this is
satisfied, e.g., when $D$ is smooth), and made
some initial progress towards (b).
One of our goals in the present article is
to establish one direction of the equivalence
in part (b) in dimension two.

To make the notion of `sufficiently small $\be$' more precise,
we introduce some terminology.
Consider a pair $(X,D)$ where $D=\sum_{i=1}^mD_i$ is a snc divisor.
Denote
\beq\label{AmplebetaEq}
\h{Amp}(X,D):=\{\be\in\RR^m_+\,:\, -K_X-\sum_{i=1}^m(1-\beta_i)D_i
\h{ \ is ample}\}.
\eeq

\begin{definition} {\rm \cite[Definition 1.1]{CR}}
\label{definition:log-Fano}{\it
We say $(X,D)$  is {asymptotically log Fano} (ALF) if
$0\in \overline{\h{\rm Amp}(X,D)}$, and
{strongly asymptotically log Fano} if
$\h{\rm Amp}(X,D)$ contains a punctured neighborhood of $0$ in $\RR^m_+\sm\{0\}$.
}
\end{definition}

When $m=1$ these two notions coincide.
Understanding which pairs $(X,D)$ admit a KEE metric with a small
angle along $D$ requires understanding the class of ALF varieties.
Recall that by Kawamata--Shokurov's Basepoint-free Theorem 
if $(X,D)$ is ALF, then $|a(K_{X}+D)|$ (for some $a\in\NN$) is 
free from base points and gives a morphism
$$
\eta\colon X\to Z,
$$
so that $Z$ is a point if and only if $D\sim -K_{X}$
\cite[Theorem 2.1]{Shokurov} (see also \cite[Theorem 1.9]{CR}), since
$\mathrm{Pic}(X)$ has no torsion 
\cite[Proposition 2.1.2]{IsPr99}.
The following conjecture, posed in our earlier work, gives a rather complete picture
concerning (b)--(c) when $D$ is smooth:

\begin{conjecture}{\rm\cite[Conjecture 1.11]{CR}}
\label{conjecture:big} Suppose that $(X,D)$ is asymptotically log
Fano manifold with $D$ smooth and irreducible.
\begin{itemize}
\item[(i)] If $\eta$ is
birational, there exist no KEE metrics for sufficiently small $\beta$.

\item[(ii)] If $\eta$ is not birational, then there exist KEE metrics
$\omega_{\beta}$ with angle $2\pi\beta$ along $D$ for all
sufficiently small $\beta>0$. Moreover, as $\beta$ tends to zero
$(X,D,\omega_\beta)$ converges in an appropriate sense to a
generalized KE metric $\omega_\infty$ on $X\setminus D$ that is
Calabi--Yau along generic fibers of $\eta$.
\end{itemize}
\end{conjecture}

This conjecture suggests that the existence problem for KEE
metrics in the small angle regime boils down to computing a single
intersection number! Namely, checking whether
$$
(K_X+D)^n=0.
$$
This is a rather far-reaching simplification
as compared to checking the much harder condition of
K-stability. Indeed, the easier direction of the
Yau--Tian--Donaldson conjecture implies that a KEE metric exists
only if the pair $(X,D)$ is log K-stable \cite{Berman}.
However, even in
dimension two, it is a very difficult
problem to check (log) K-stability as it involves, in theory,
computing the Futaki invariant of an infinite number of
test configurations.

\subsection{Flop-slope stability and non-existence}

When $n=2$, Conjecture \ref{conjecture:big} (i) amounts to:

\begin{conjecture}{\rm \cite[Conjecture~1.6]{CR}}
\label{conjecture:main} Let $S$ be a smooth surface, and let $C$
be a smooth  irreducible curve on $S$. Suppose that $(S,C)$ is asymptotically
log del Pezzo. Then $S$ admits KEE metrics
with angle $\beta$ along $C$ for all sufficiently small $\beta$
 only if $(K_{S}+C)^2=0$.
\end{conjecture}

Our main result is a verification
of Conjecture~\ref{conjecture:main}.

\begin{theorem}
\label{MainThm}
 Let $S$ be a smooth surface, and let $C$ be a
smooth irreducible curve on $S$. Suppose that $(S,C)$ is asymptotically log
del Pezzo and $(K_{S}+C)^2\not=0$. Then $S$ does not admit
 KEE metrics with angle $\beta$ along $C$ for
all sufficiently small $\beta$. Moreover, this statement holds
for all $\be\in \h{\rm Amp}(M,D)$ for which \eqref{FutakiLongEq} is negative.
\end{theorem}

In other words, we give a completely elementary and verifiable
criterion that is equivalent to log K-unstability in the small
angle regime. The proof involves a modification of the notion of
slope stability due to Ross and Thomas \cite{RT1}, where we
additionally perform flops on the deformation to the
normal cone.
In Li--Xu \cite{LiXu} it was shown that the generalized Futaki invariant
decreases under certain modifications and our construction
is partly inspired by those general results, although we do not make
use of them.
This construction using flops
occupies most of this article,
and we believe it is of independent interest.

This {\it flop-slope }$\!\!$ 
construction is essential to the proof of Theorem \ref{MainThm}
since for asymptotically logarithmic del Pezzo surfaces the
more traditional obstructions of Matsushima, Futaki, and
Ross--Thomas \cite{Matsushima,Futaki,RT1} are not sufficient, as examples in this article and in \cite{CR} show. We expect the method developed in
this article to yield
many more new examples of non-existence in
different settings and in
 higher dimensions.

We remark that the converse to Conjecture~\ref{conjecture:main}
is open: we refer to \cite[\S9]{R14} for a discussion of partial
results.

\subsection{Organization}

In \S\ref{PrelimSec} we review some preliminaries: the intersection-theoretic
formula for the generalized Futaki invariant, (log) slope stability,
and also derive some related useful formulas for asymptotically log del Pezzo
surfaces. In \S\ref{section:maeda-class} we apply these formulas to prove
Theorem \ref{MainThm} for the simplest
subclass of \aldp surfaces: the Maeda class for which $0\in\h{\rm Amp}(M,D)$.
Section \ref{section:flops} is the heart of the article, and contains
our
modification of slope stability, which we call flop-slope stability.
The main result here is Proposition \ref{proposition:Futaki-flop-slop}
that gives a formula for the Futaki invariant for
the flopped test configuration. Some technical intersection-theoretic
result needed here is proved in Appendix A.
The proof of Theorem \ref{MainThm} is then carried out in \S\ref{section:proof}.
In \S\ref{ExampSec} we collect some further examples.

\chardef\inodot="10

\subsection*{Acknowledgements}
The authors thank
J. Mart\'\inodot nez-Garc\'\inodot a 
and R. Thomas for comments.
The authors are grateful to R.J. Berman and Chalmers
University of Technology for the hospitality and financial
support during Summer 2014.
The NSF supported this research through grants DMS-1206284,1515703.
IAC was also supported by grant 15-11-20036 of the 
Russian Science Foundation.
YAR was also supported by a Sloan Research Fellowship.

\section{Preliminaries}
\label{PrelimSec}
\subsection{Generalized Futaki invariant}
\label{section:DF}

Let ${{\beta}}=(\beta_1,\ldots,\beta_m)\in(0,1]^m$ be a
vector. Let $X$ be a normal $\mathbb{Q}$-factorial variety (of
complex dimension $n$), let $D=\sum_{i=1}^mD_i$ be a divisor
(where the $D_i$ are distinct $\mathbb{Q}$-Cartier prime Weil
divisors) on $X$, and let $L_{{\beta}}$ be an ample
$\mathbb{R}$-divisor on $X$ (that a priori may depend on
${{\beta}}$). Put
$$
D_{{\beta}}=\sum_{i=1}^m(1-\beta_i)D_i.
$$
Let
$(\mathcal{X},\mathcal{L}_{{\beta}},\mathcal{D},{{\beta}})$
be a quadruple consisting of a normal $\mathbb{Q}$-factorial
variety $\mathcal{X}$ of dimension $n+1$, equipped with a flat
surjective map
$p\colon \mathcal{X}\to\mathbb{P}^1$,
$\mathbb{R}$-divisor $\mathcal{L}_{{\beta}}$
\footnote{
We do {\it not} assume $\mathcal{L}_{{\beta}}$
is $p$-ample in the definition of
$\mathrm{F}(\mathcal{X},\mathcal{L}_{{\beta}},\mathcal{D},{{\beta}}).$
An $\mathbb{R}$-Cartier divisor $A$ on
$\mathcal{X}$ is  $p$-ample (resp., $p$-big) if $A+p^\star B$ is
ample (resp., big) for some divisor $B$ on $\mathbb{P}^1$.}; a divisor
$\mathcal{D}=\sum_{i=1}^m\mathcal{D}_i$ (where the $\mathcal{D}_i$
are distinct $\mathbb{Q}$-Cartier prime Weil divisors) on
$\mathcal{X}$. Suppose that all fibers of $p$ except the fiber
over $[0:1]$ (which we call the central fiber) 
are isomorphic to $X$, and the divisors
$\mathcal{L}_{{\beta}}$ and $\mathcal{D}_i$ restricted to
these fibers are $L_{{\beta}}$ and $D_i$, respectively. Thus,
$\mathrm{Supp}(\mathcal{D})$ does not contain components of the
fibers of $p$ (if it did, $\mathcal{D}_i$ restricted to
different fibers would be different, but we assume the restriction is always
the same, namely, $D_i$), 
and so in particular it does not contain components
of the central fiber.
The generalized Futaki invariant is
\begin{multline}
\label{FDefEq}
\mathrm{F}(\mathcal{X},\mathcal{L}_{{\beta}},\mathcal{D},{{\beta}}):=n\frac{-(K_X+\sum_{i=1}^m(1-\beta_i)D_i).L_{{\beta}}^{n-1}}{L_{{\beta}}^n}\mathcal{L}_{{\beta}}^{n+1}\\
+(n+1)\Big(K_{\mathcal{X}}-p^\star(K_{\mathbb{P}^1})+\sum_{i=1}^m(1-\beta_i)\mathcal{D}_i\Big).\mathcal{L}_{{\beta}}^n.%
\end{multline}
Whenever the triple
$(\mathcal{X},\mathcal{L}_{{\beta}},\mathcal{D})$ is a test
configuration in the sense of Tian \cite{Tian1997} and Donaldson
\cite{Don2002}, then
$F(\mathcal{X},\mathcal{L}_{{\beta}},\mathcal{D})$ equals its
Futaki invariant in the sense of Ding--Tian or Donaldson
\cite{DingTian,Don2002,WangXW,Odaka,Berman,LiXu,T15}. If
\begin{equation}
\label{equation:L-beta-K}
L_{{\beta}}\sim_{\mathbb{R}}-(K_X+\sum_{i=1}^m(1-\beta_i)D_i)
\end{equation}
(so that $(X,\sum_{i=1}^m(1-\beta_i)D_i)$ is a log Fano variety),
the formula for
$\mathrm{F}(\mathcal{X},\mathcal{L}_{{\beta}},\mathcal{D},{{\beta}})$
simplifies to
\begin{equation}
\label{equation:intersection-formula}
\mathrm{F}(\mathcal{X},\mathcal{L}_{{\beta}},\mathcal{D},{{\beta}})=n\mathcal{L}_{{\beta}}^{n+1}+(n+1)\bigg(K_{\mathcal{X}}
-p^\star K_{\mathbb{P}^1}+\sum_{i=1}^m(1-\beta_i)\mathcal{D}_i\bigg).\mathcal{L}_{{\beta}}^n.%
\end{equation}

We recall the following result:

\begin{theorem}{\rm\cite[Theorem 4.8]{Berman}}
\label{theorem:Berman} Suppose that $X$ is smooth,
$D=\sum_{i=1}^mD_i$ is a simple normal crossing divisor, and
\eqref{equation:L-beta-K} holds. Let
$(\mathcal{X},\mathcal{L}_{{\beta}},\mathcal{D},{{\beta}})$
be a test configuration for $(X,L_{{\beta}},D_{{\beta}})$.
Assume that $\mathcal{L}_{{\beta}}$ is $p$-ample.
If $\mathrm{F}(\mathcal{X},\mathcal{L}_{{\beta}},\mathcal{D},{{\beta}})<0$,
then $(X,D,{\beta})$ does not admit a KEE metric.
\end{theorem}

The simplest possible case (beyond a product configuration) 
when we can effectively apply this
theorem is when $X$ is smooth  and the triple
$(\mathcal{X},\mathcal{L}_\beta,\mathcal{D})$ is a very particular
test configuration obtained via deformation to the normal cone of
a smooth subvariety in $X$. This construction is originally due to
Ross--Thomas \cite{RT1}. We now turn to describe it.

\subsection{Slope stability}
\label{section:slope-stability}

Let $X$ be a smooth variety, and let $Z$ be a smooth subvariety in
$X$. Consider the blow-up of $Z\times\{[0:1]\}$ in
$X\times\mathbb{P}^1$. We denote the resulting space (of complex
dimension $n+1$) by $\mathcal{X}$ and denote the blow-down map by
$\pi_Z$. Denote the $\pi_Z$-exceptional divisor by $E_Z$. Let
$p_{\mathbb{P}^1}\colon
X\times\mathbb{P}^1\rightarrow\mathbb{P}^1$ and $p_X\colon
X\times\mathbb{P}^1\rightarrow X$ denote the natural projections.

Put 
$$
p:=p_{\mathbb{P}^1}\circ\pi_Z.
$$
The morphism
$p\colon\mathcal{X}\to\mathbb{P}^1$ is flat 
\cite[Proposition~9.7]{Har77}. Its fibers over every point that is
different from $[0:1]$ are isomorphic to $X$. The fiber
$\mathcal{X}_0$ over $[0:1]\in\mathbb{P}^1$ is the union $E_Z\cup
X_0$, where $X_0$ is the proper transform of $X\times\{[0:1]\}$,
and 
\beq
\lb{EZEq}
E_Z=\mathbb{P}(\nu_Z\oplus\mathcal{O}_Z)
\eeq
is a smooth ruled variety. 
Here $\nu_Z$ denotes the normal bundle of $Z$ in $X$, and
$\mathcal{O}_Z$ denotes the trivial line bundle over $Z$. Of
course, $\nu_Z\oplus\mathcal{O}_Z$ is the normal bundle of
$Z\times\{[0:1]\}$ in $X\times\mathbb{P}^1$. Note that $X_0$ is
the blow-up of $X$ at $Z$. Thus, if $Z$ is a divisor in $X$, then
$X_0$ is simply a copy of $X$.

Denote by $\pi_0$ the morphism 
$
p_X\circ\pi_Z|_{X_0}:X_0\rightarrow X$, which is just
the blow-down map of $Z$ in $X$.
In
fact, 
$E_Z$ intersect $X_0$ exactly at the
exceptional locus of $\pi_0$ (here we slightly abuse language,
since when $Z$ is a divisor, this locus is not exceptional, but is
just a copy of $Z$, the proper transform of $Z$).

Let ${{\beta}}=(\beta_1,\ldots,\beta_m)\in(0,1]^m$ be a
vector, and let $L_{{\beta}}$ be an  ample  $\mathbb{R}$-divisor
on $X$ that may depend on the vector ${{\beta}}$. Put
\beq\label{LbetaEq}
\mathcal{L}_{{\beta,c}}:=(p_X\circ\pi_Z)^\star L_{{\beta}}-cE_Z%
\eeq
for some $c>0$. Recall the definition of the Seshadri constant of
$(X,Z)$ with respect to $L_\beta$,
\beq\label{espilonXZLEq}
\epsilon(X,Z,L_{{\beta}})=\sup\big\{c>0\,:\, \pi_0^\star(L_{{\beta}})-cE_Z|_{X_0}
\h{\rm \ is ample}\big\}.%
\eeq
Thus, if $c\ge\epsilon(X,Z,L_{{\beta}})$, then
$\mathcal{L}_{{\beta}}$ is not $p$-ample.
The following is a special case of \cite[Lemma~4.1]{RT1}. We
give a simple direct proof for the reader's convenience.
We make use of the following simple fact more than once in this
article, so we record it here:
\beq
\lb{intersectfiberEq}
\h{if $C$ is a curve contained in the central fiber then
$C.S_0=-C.E_Z$.}
\eeq
Indeed, since $C$ is contained in a fiber of $p:{\mathcal X}\ra\PP^1$
and $S_0\cup E_Z$ is such a fiber (the central fiber) then
$C.(S_0+E_Z)=0$.

\begin{lemma}{\rm}
\label{lemma:p-ample} Suppose that $c\in(0,
\epsilon(X,Z,L_{{\beta}}))$. Then $\mathcal{L}_{{\beta},c}$
is $p$-ample.
\end{lemma}

\begin{proof}
Since $L_{{\beta}}$ is ample, by Kleiman's criterion 
there is a positive constant $\gamma_0$
depending only on $L_{{\beta}}$ such that
$$
L_{{\beta}}. C\ge\gamma_0
$$
for every curve $C$ in $X$. Similarly, there is a positive constant
$\gamma_1$ depending on $L_{{\beta}}$ and $c$ alone such
that
$$
\big(\pi_0^\star(L_{{\beta}})-cE_Z|_{X_0}\big). C\ge\gamma_1%
$$
for every curve $C\subset X_0$, because
$c<\epsilon(X,Z,L_{{\beta}})$.

Put 
$$
\gamma:=
\min\{c,\gamma_0,\gamma_1\}.
$$ 
We claim that
$\mathcal{L}_{{\beta},c}. C\ge\gamma$ for every curve
$C\subset\mathcal{X}$ such that $p(C)$ is a point. The latter
implies $p$-ampleness of the divisor $\mathcal{L}_{{\beta},c}$.

Let $C$ be a curve in $\mathcal{X}$ such that $p(C)$ is a point
(so that $C$ lies in some fiber).
If $C$ is not in the central fiber $E_Z\cup X_0$, then
$$
\mathcal{L}_{{\beta},c}. C
=L_{{\beta}}. p_X\circ\pi_Z(C)\ge\gamma_0\ge\gamma.%
$$
If $C$ is in the central fiber and is contracted by $\pi_Z$ to a
point, i.e., $C$ is contained in a fiber of $E_Z\mapsto Z$ (a
$\PP^{n-1}$ bundle), then
$
\mathcal{L}_{{\beta},c}. C=-cE_Z. C\ge
c\ge\gamma,
$
since $-E_Z. C\ge 1$ in this case
as $-E_Z$ restricts to the hyperplane bundle on each fiber. 
If $C\subset E_Z$,
$C\not\subset X_0$ and $C$ is not contracted by $\pi_Z$ to a
point, then using \eqref{intersectfiberEq},
$$
\mathcal{L}_{{\beta},c}.
C=L_{{\beta}}.p_X\circ\pi_Z(C)-cE_Z. C
=
L_{{\beta}}.p_X\circ\pi_Z(C)+cX_0. C
\ge 
L_{{\beta}}.p_X\circ\pi_Z(C)\ge\gamma_0\ge\gamma.%
$$
If $C\subset X_0$, then
$$
\mathcal{L}_{{\beta},c}.
C=\big((p_X\circ\pi_Z)^\star(L_{{\beta}})-cE_Z\big)\vert_{X_0}.
C=\big(\pi_0^\star(L_{{\beta}})-cE_Z|_{X_0}\big).\pi_Z(C)\ge\gamma_1\ge\gamma,
$$
concluding the proof.
\end{proof}

Let $D=\sum_{i=1}^mD_i$ be a simple normal crossing divisor on
$X$, where the $D_i$ are distinct smooth prime divisors on $X$.
For every $c\in (0,\epsilon(X,Z, L_{{\beta}}))$,
$(\mathcal{X},\mathcal{L}_{{\beta},c},\mathcal{D}_{{\beta}})$
is a test configuration for $(X,L_{{\beta}},D_{{\beta}})$.
We assume that $\mathcal D_i$ is the proper transform
of $D_i\times\PP^1$ in $\mathcal X$.
Recall the following definition due to Ross--Thomas and Li--Sun.

\begin{definition}
\label{definition:slope-unstable} The triple  
$(X,L_{{\beta}},D_{{\beta}})$ is slope unstable with respect to $Z$ if
$F(\mathcal{X},\mathcal{L}_{{\beta},c},\mathcal{D},{{\beta}})<0$
for some $c\in(0,\epsilon(X,Z, L_{{\beta}}))$.
\end{definition}

Note that according to \eqref{espilonXZLEq}
and Lemma \ref{lemma:p-ample}, the assumption on $c$ in
Definition \ref{definition:slope-unstable} guarantees that
Theorem~\ref{theorem:Berman} is applicable. 

\begin{corollary}
\label{corollary:slope-unstable} If 
$(X,L_{{\beta}},D_{{\beta}})$ is slope unstable with respect to $Z$, then 
$(X,D,\beta)$
does not admit a KEE metric.
\end{corollary}

The importance of this corollary is that the number
$\mathrm{F}(\mathcal{X},\mathcal{L}_{{\beta},c},\mathcal{D},{{\beta}})$
is readily computable for the test configuration described in this subsection
(compared to a general test configuration).

\begin{remark}
\label{remark:slope}
{\rm In all cases we considered so far, if
$F(\mathcal{X},\mathcal{L}_{{\beta},c},\mathcal{D},{{\beta}})<0$
for some $c\in(0,\epsilon(X,Z, L_{{\beta}}))$, then
$F(\mathcal{X},\mathcal{L}_{{\beta},c},\mathcal{D},{{\beta}})<0$
for $c=\epsilon(X,Z, L_{{\beta}})$.
}\end{remark}

In the next section, we compute
$F(\mathcal{X},\mathcal{L}_{{\beta},c},\mathcal{D},{{\beta}})$
in a particular situation.

\subsection{Slope stability for logarithmic surfaces}
\label{section:slope-stability-surface}

Let us use the notation and assumptions of
\S\ref{section:slope-stability}. Suppose, in addition, that $D$ is
a smooth curve in a smooth surface $X$, i.e.,  $m=1, n=2$, 
and $D=D_1, \mathcal{D}=\mathcal{D}_1$, and $Z$ is a smooth
curve in $X$. For transparency, we put $S=X$, $S_0=X_0$,
$p_X=p_S$, $C=D=D_1$, $\beta=\beta_1$,
and $D_\beta=(1-\beta)C$. Then $\mathcal{X}$
is a threefold, and the fiber over $[0:1]\in\mathbb{P}^1$ is the
union of two surfaces $E_Z\cup S_0$, where $S_0$ is the proper
transform of the fiber of $p_{\mathbb{P}^1}$ over $[0:1]$. Since
$C$ is a curve, we have $S_0\cong S$.
Note that
the exceptional divisor
$E_Z\cong\mathbb{P}(\nu_Z\oplus\mathcal{O}_Z)$ is a smooth ruled
surface, where $\nu_Z$ denotes the normal bundle of $Z$ in $S$,
and $\mathcal{O}_Z$ denotes the trivial line bundle over $Z$.

In the case when $L_\beta\sim_{\mathbb{R}}-K_S-(1-\beta)C$, 
there is an
explicit formula for
$F(\mathcal{X},\mathcal{D},\mathcal{L}_\beta,\beta)$.
First, recall some intersection formulas.

\blemma
\lb{IntersectLemma}
One has,  
\beq
\lb{firsteqIntLemmaEq}
E_Z^3=-\mathrm{deg}(N_{Z/\mathcal{X}})=-Z^2,
\eeq
and
\beq
\lb{secondeqIntLemmaEq}
(p_S^\star L_\beta)^3=0,  \qquad 
((p_S\circ\pi_Z)^\star L_\beta).E_Z^2=-(p_S^\star L_\beta).Z=-L_\beta.Z.
\eeq 
\elemma

\bpf
The first equality in \eqref{firsteqIntLemmaEq} follows from \cite[p. 608]{GH}
while the second equality follows from the fact that 
$N_{Z/\mathcal{X}}$ decomposes as $\mathcal{O}_Z(-1)\oplus \mathcal{O}_Z$.
Since $E_Z$ is the projectivization of $N_{Z/\mathcal{X}}$, the previous
decomposition also implies \eqref{secondeqIntLemmaEq} as $\pi_Z(E_Z)=Z$.
\epf

\begin{proposition}
\label{proposition:formula} Suppose that
\eqref{equation:L-beta-K} holds.
Then,
$$
F(\mathcal{X},\mathcal{L}_{\beta,c},\mathcal{D},\beta)=
\left\{%
\aligned
&\Big(6\beta c-3c^2\Big)L_\beta.Z+\Big(2c^3-3c^2\beta\Big)Z^2\ 
\ \ \text{\rm if}\ Z=C,\\%
&\Big(6c-3c^2\Big)L_\beta.Z+\Big(2c^3-3c^2\Big)Z^2\ \ \ \text{\rm if}\ Z\ne C.\\%
\endaligned\right.%
$$
\end{proposition}

\begin{proof}
First, using Lemma \ref{IntersectLemma},
$$
\baeq
\mathcal{L}_{\beta,c}^3
&=
\Big((p_S\circ\pi_Z)^\star L_\beta-cE_Z\Big)^3
\\
&=
((p_S\circ\pi_Z)^\star L_\beta)^3-3c\Big((p_S\circ\pi_Z)^\star L_\beta\Big)^2.E_Z+3c^2\Big((p_S\circ\pi_Z)^\star L_\beta\Big).E_Z^2-c^3E_Z^3
\\
&=
(p_S^\star L_\beta)^3+3c^2\Big((p_S\circ\pi_Z)^\star L_\beta\Big).E_Z^2-c^3E_Z^3=3c^2\Big((p_S\circ\pi_Z)^\star L_\beta\Big).E_Z^2-c^3E_Z^3
\\
&=-3c^2(p_S^\star L_\beta).Z-c^3E_Z^3
=-3c^2L_\beta.Z+c^3Z^2.
\eaeq
$$

For the second term in \eqref{equation:intersection-formula}, suppose first that
$Z=C$ (this is only used in the second line in computing $\mathcal D$). Using Lemma \ref{IntersectLemma} and the formula for the canonical bundle and a general divisor 
under a blow-up \cite[p. 187, 476]{GH},
$$
\baeq
\Big(K_{\mathcal{X}}
&-p^\star K_{\mathbb{P}^1}
+(1-\beta)\mathcal{D}\Big).\mathcal{L}_{\beta,c}^2
\cr
&=
\Big(
(p_S\circ\pi_Z)^\star K_{S}
+
E_Z
+
(1-\beta)\pi_Z^\star(C\times\mathbb{P}^1)
-
(1-\beta)E_Z
\Big)
.
\Big(
(p_S\circ\pi_Z)^\star L_\beta
-cE_Z\Big
)^2
\cr
&=
\Big(
(p_S\circ\pi_Z)^\star K_{S}
+
(1-\beta)(p_S\circ\pi_Z)^\star C
+
\beta E_Z
\Big)
.
\Big(
((p_S\circ\pi_Z)^\star L_\beta)^2
-
2cE_Z.(p_S\circ\pi_Z)^\star L_\beta
+
c^2E_Z^2\Big)
\cr
&=
c^2(p_S\circ\pi_Z)^\star K_{S}.E_Z^2
+
(1-\beta)c^2(p_S\circ\pi_Z)^\star C.E_Z^2
-
2\beta c (p_S\circ\pi_Z)^\star L_\beta.E_Z^2
+
\beta c^2 E_Z^3
\cr
&=
-
c^2p_S^\star K_{S}.Z
-
(1-\beta)c^2 \pi_Z^\star C.Z
+
2\beta c L_\beta.Z
-
\beta c^2Z^2
\cr
&=
-
c^2K_{S}.Z
-
(1-\beta)c^2 C.Z
+
2\beta c L_\beta.Z
-
\beta c^2 Z^2.
\cr
\eaeq
$$
If $Z\ne C$, then
$
\mathcal{D}
=
(1-\beta)\pi_Z^\star(C\times\mathbb{P}^1)
$,
so the previous calculation gives
$$
\big(K_{\mathcal{X}}
-p^\star K_{\mathbb{P}^1}
+
(1-\beta)\mathcal{D}
\big)
.
\mathcal{L}_{\beta,c}^2
=
-
c^2K_{S}.Z
-
(1-\beta)c^2 C.Z
+
2c L_\beta.Z
-
c^2 Z^2.
$$

Thus, if $Z=C$, we have
$$
F(\mathcal{X},\mathcal{L}_{\beta,c},\mathcal{D},\beta)
=
2[-3c^2L_\beta.Z+c^3Z^2]
+3[-c^2K_{S}.Z+2\beta cL_\beta.Z-c^2 Z^2],
$$
while if $Z\ne C$, we have
$$
F(\mathcal{X},\mathcal{L}_{\beta,c},\mathcal{D},\beta)
=
2[-3c^2L_\beta.Z+c^3Z^2]
+
3
[
-
c^2K_{S}.Z
-
(1-\beta)c^2C.Z 
+
2cL_\beta.Z
-
c^2 Z^2].
$$
Plugging in \eqref{equation:L-beta-K} now yields the desired
formulas.
\end{proof}

In the next section, we will show how to apply
Proposition~\ref{proposition:formula} to compute
$F(\mathcal{X},\mathcal{D},\mathcal{L}_\beta,\beta)$ in some cases
(cf. Li--Sun \cite[Proposition~3.15, Example~3.16]{LS}). Before doing so, we illustrate
with a simple example.  

\begin{example}
\label{example:F1-rational}
{\rm
Suppose that $S=\mathbb{F}_1$ and $C$
is a smooth rational curve in $|E+F|$, where $F$ is a fiber of the
natural projection $S\to\mathbb{P}^1$, and $E$ is the unique
$-1$-curve in $S$. Then $L_\beta$ is ample for every
$\beta\in(0,1]$. The automorphism
group of the pair $(S,C)$ is reductive
\cite[Proposition 7.1]{CR} so the edge version of 
Matsushima's obstruction
\cite[Theorem 1.12]{CR}
is not applicable. If $Z=C$ or
$Z=E$, then $\epsilon(S,L_\beta,Z)=1+\beta$. 
In addition, if
$Z=C$, Proposition~\ref{proposition:formula} gives
$$
F(\mathcal{X},\mathcal{L}_{\beta,c},\mathcal{D},\beta)=2(1+\beta)(\beta^2+2\beta-2)%
$$
for $c=1+\beta$, so
$F(\mathcal{X},\mathcal{L}_{\beta,c},\mathcal{D},\beta)<0$ for
$\beta<\sqrt{3}-1$. Similarly, if $Z=E$,
Proposition~\ref{proposition:formula} gives
$$
F(\mathcal{X},\mathcal{L}_{\beta,c},\mathcal{D},\beta)
=(1+\beta)(2-\be^2-2\beta).%
$$
for $c=1+\beta$, so
$F(\mathcal{X},\mathcal{L}_{\beta,c},\mathcal{D},\beta)<0$ for
all $\beta>\sqrt3-1$.
}
\end{example}

\section{Maeda's class}
\label{section:maeda-class}

Let $C$ be a smooth curve on a smooth surface $S$.
Suppose that $(S,C)$ is asymptotically log del Pezzo. Put
$$
L_\beta\sim_{\mathbb{R}}-K_{S}-(1-\beta)C,
$$ 
where
$\beta\in(0,1]$ is such that $L_\beta$ is ample.

The following result proves in a unified manner than whenever $-K_S-C$
is ample, Conjecture \ref{conjecture:main} holds. Alternatively, this result
also follows by combining \cite[Proposition 7.1]{CR} with 
\cite{LS} and Example \ref{example:F1-rational}.

\begin{prop}
\label{theorem:Maeda} Suppose that $-K_{S}-C$ is ample. 
Then $(S,C,\be)$ does not admit a KEE metric
for all sufficiently small $\beta$.
\end{prop}

\bremark
\lb{RationalRemark}
By \cite[Corollary 2.3]{CR} $C$ is rational.
\eremark

\begin{proof}
Pick any positive
$\gamma<\epsilon(S,Z,-K_{S}-C)$. By definition,
$$
-K_{S}-C\sim_{\mathbb{R}}\gamma Z+H
$$
for some ample $\mathbb{R}$-divisor $H$.
Letting $Z:=C$ then
$$
L_\beta\sim_{\mathbb{R}}-K_{S}-C+\beta
C\sim_{\mathbb{R}}\Big(\gamma+\beta\Big)C+H,
$$
which implies that
$\epsilon(S,L_\beta,Z)\ge\gamma+\beta>\gamma$. 

Pick some $c\in(0,\gamma]$. Let us use notation and assumptions of
\S\ref{section:slope-stability-surface}. Then $\mathcal{L}_{\beta,c}$
is $p$-ample by Lemma~\ref{lemma:p-ample}.
By Remark \ref{RationalRemark},
$L_\beta.C=-(K_S+(1-\be)C).C=2+\be C^2$. 
Therefore, using
Proposition~\ref{proposition:formula},
with $c=\gamma$,
$$
\baeq
\mathrm{F}(\mathcal{X},\mathcal{L}_{\beta,c},\mathcal{D},\beta)
&=
-3\gamma^2L_\beta. C+2\gamma^3C^2+\beta(6\gamma L_\beta.C-3\gamma^2C^2)
\cr
&=-\gamma^2L_\beta. C-2\gamma^2(L_\beta-\gamma C).C+\beta
(6\gamma L_\beta.C-3\gamma^2C^2)
\cr
&=-\gamma^2(2+\beta C^2)-2\gamma^2(L_\beta-\gamma C).C
+\beta(6\gamma L_\beta.C-3\gamma^2C^2)\cr
&<-2\gamma^2+\beta(6\gamma L_\beta.C-4\gamma^2C^2).
\eaeq
$$
so $\lim_{\beta\to
0^+}\mathrm{F}(\mathcal{X},\mathcal{L}_{\beta,c},\mathcal{D},\beta)\le
-2\gamma^2<0$.
Thus, Theorem~\ref{theorem:Berman} implies the desired result.
\end{proof}

\begin{remark}
\label{remark:ampleness-Maeda} One cannot drop the ampleness
condition in Proposition~\ref{theorem:Maeda}. Indeed, if $-K_{S}-C$
is not ample, then it follows from the classification in \cite{CR}
and Lemma \ref{lemma:Seshadri-constants}  
(i) below
that $\epsilon(S,L_\beta,C)\le\beta$ so the arguments used
in the proof of Proposition~\ref{theorem:Maeda} are no longer valid.
\end{remark}

\section{Flop-slope stability}
\label{section:flops}

We follow the notation and assumptions of \S\ref{section:slope-stability-surface}. 
In addition, denote by $O_1,\ldots,O_r$,
distinct points on the curve $Z$, and let 
$$
\pi_O\colon S'\to S
$$ 
be the blow-up of the union of these points,
whose exceptional curves are 
$$
C'_1,\ldots, C'_r\subset S',
$$ 
with $\pi_O(C'_i)=O_i$.
Denote by 
$$
C', Z'\subset S'
$$ 
the $\pi_O$-proper transforms of the curves $C,Z\subset S$,
respectively. 
Let 
$$
p_{S'}
\colon
S'\times\mathbb{P}^1\to S',
\q
p_{\mathbb{P}^1}'
\colon
S'\times\mathbb{P}^1\to\mathbb{P}^1,
$$ 
be the natural projections. 
Put
\beq
\lb{LiEq}
L_i:=\{O_i\}\times\mathbb{P}^1\subset S\times\PP^1,
\eeq
and let
$$
\pi_L
\colon
S'\times\mathbb{P}^1\to S\times\mathbb{P}^1
$$ 
be the blow-up of the union of the
smooth disjoint curves $L_1,\ldots, L_r$. From now on, by abuse of
notation, we identify $S$ and $S'$ with the fibers of
$p_{\mathbb{P}^1}$ and $p_{\mathbb{P}^1}'$ over the point 
$[0:1]\in\PP^1$,
respectively. The blow-up $\pi_O\colon S'\to S$ is induced
by the blow-up $\pi_L$. In sum, there exists a commutative diagram:
$$
\xymatrix{
S'\ar@{^{(}->}[d]\ar@{->}[rr]^{\pi_O}&&S\ar@{^{(}->}[d]\\%
S'\times\mathbb{P}^1\ar@{->}[dr]_{p_{\mathbb{P}^1}'}\ar@{->}[rr]^{\pi_L}&& S\times\mathbb{P}^1\ar@{->}[ld]^{p_{\mathbb{P}^1}}\\%
&\mathbb{P}^1&}
$$

Let 
$$
\pi_{Z'}\colon\mathcal{X}'\to S'\times\mathbb{P}^1
$$ 
be the blow-up of the curve $Z'\subset S'\subset S'\times\mathbb{P}^1$, and
let 
$$
E_{Z'}\subset \mathcal{X}'
$$ 
be the $\pi_{Z'}$-exceptional divisor. 
If $Z$ is rational, then $E_{Z'}\cong\mathbb{F}_k$
where $k=|{Z'}^2|$ (to see this recall \eqref{EZEq}).
Denote by 
$$
S'_{0}\subset \mathcal{X}'
$$ 
the $\pi_{Z'}$-proper transform of the surface $S'\subset S'\times\mathbb{P}^1$. 
Put 
$$
p':=p_{\mathbb{P}^1}'\circ\pi_{Z'}:\mathcal{X}'\ra \PP^1.
$$ 
Then,
\beq\lb{SprzeroEq}
S'_0\cong S'
\eeq
and 
$$
S'_0\cup E_{Z'}
$$ 
is the fiber of $p'$ over the
point $[0:1]$ (the ``central fiber"). 
Denote by 
$$
C_1,\ldots,C_r\subset \mathcal{X}'
$$ 
the $\pi_{Z'}$-proper transform on
of the curves $C'_1,\ldots,C'_r\subset S'\subset S'\times\PP^1$, respectively.
Then $C_i\cong\mathbb{P}^1$. 

\begin{lemma}
\label{lemma:normal-bundle} The normal bundle of $C_i$ in
$\mathcal{X}'$ is isomorphic to
$\mathcal{O}_{\mathbb{P}^1}(-1)\oplus
\mathcal{O}_{\mathbb{P}^1}(-1)$.
\end{lemma}

\begin{proof}
Since $C_i$ is rational, by Grothendieck's lemma \cite[p. 516]{GH}
$N_{C_i|\mathcal{X}'}=\mathcal{O}(a)\oplus\mathcal{O}(b)$. Thus,
$$
0\rightarrow T_{C_i}\rightarrow T_{\mathcal{X}'}|_{C_i}
\rightarrow \mathcal{O}(a)\oplus\mathcal{O}(b) \rightarrow 0,
$$
implies (considering the first Chern classes) that 
\beq
\lb{intersectionKVCEq}
a+b+2-2g(C_i)
=
c_1(\mathcal{X}').C_i=-K_{\mathcal{X}'}.C_i.
\eeq
Note that
$C_i.E_{Z'}=1$ since $C'_i$ and $Z'$ intersect transversally
at one point downstairs (in $S'\subset S'\times\PP^1$).
In addition, 
$K_{\mathcal{X}'}=\pi_{Z'}^\star K_{S'\times\mathbb{P}^1}+E_{Z'}$. Thus,
$$
K_{\mathcal{X}'}.C_i
=\pi_{Z'}^\star K_{S'\times\mathbb{P}^1}.C_i+1
=K_{S'\times\mathbb{P}^1}.C_i'+1
=K_{S'\times}.C_i'+1=2g(C_i')-2-(C_i')^2+1=0.
$$
Thus, from \eqref{intersectionKVCEq} we conclude that $a+b=-2$.
Next,
\beq
\lb{sesEq}
0\rightarrow N_{C_i|S'_{0}}\rightarrow N_{C_i|\mathcal{X}'}
\rightarrow N_{S'_0|\mathcal{X}'}|_{C_i} \rightarrow 0.
\eeq
Observe that $N_{C_i|S'_0}=\mathcal{O}_{\PP^1}(-1)$ since $C_i$ is a $-1$-curve in $S'_0$. Thus, taking first Chern classes
and using the previous paragraph, we must have
$N_{S'_0|\mathcal{X}'}|_{C_i}=\mathcal{O}_{\PP^1}(-1)$. The long
exact sequence associated to \eqref{sesEq} gives
$$
0= H^0(\PP^1,\mathcal{O}_{\PP^1}(-1))
\ra H^0(\PP^1,\mathcal{O}_{\PP^1}(a)\oplus \mathcal{O}_{\PP^1}(b))
\ra 
H^0(\PP^1,\mathcal{O}_{\PP^1}(-1))=0,
$$
implying that $a,b<0$; thus, $a=b=-1$.
\end{proof}

Thus, as described in Appendix~\ref{appendix}, we can simultaneously flop 
the curves $C_1,\ldots,C_r\subset \mathcal{X}'$. 
Denote this composition of simple
flops by $f\colon \mathcal{X}'\to
\hat{\mathcal{X}}'$. Moreover, there exists a surjective morphism
$$
\hat{p}'\colon \hat{\mathcal{X}}'\to\mathbb{P}^1
$$ 
that makes the diagram
$$
\xymatrix{
&&\mathcal{X}'\ar@{->}[dd]_{p'}\ar@{->}[dl]_{\pi_{Z'}}\ar@{-->}[rr]^{f}&& \hat{\mathcal{X}}'\ar@{->}[dd]^{\hat{p}'}\\%
&S'\times\mathbb{P}^1\ar@{->}[dr]_{p'_{\mathbb{P}^1}}&&&\\%
&&\mathbb{P}^1\ar@{=}[rr]&&\mathbb{P}^1&}
$$ %
commute. Note that $\hat{p}'$ is flat 
\cite[Proposition~9.7]{Har77}. 
Let us show how to obtain $\hat{\mathcal{X}}'$ even more
explicitly by blowing up the threefold $\mathcal{X}$. This will
also show that $\hat{\mathcal{X}}'$ is projective.

\begin{remark}
\label{remark:Fm} Recall from
\S\ref{section:slope-stability-surface} that we have a blow up
$\pi_Z\colon\mathcal{X}\to S\times\mathbb{P}^1$ of the curve
$Z\subset S\subset S\times\mathbb{P}^1$, and we denoted the
$\pi_Z$-exceptional divisor by $E_{Z}$. If $Z$ is rational, then
$E_{Z}\cong\mathbb{F}_{|Z^2|}$.
\end{remark}

Denote by 
$$
\tilde{L}_1,\ldots,\tilde{L}_r\subset \mathcal{X}
$$ 
the $\pi_Z$-proper transforms
of the curves $L_1,\ldots,L_r$ (defined in \eqref{LiEq}).
Then, each $\tilde{L}_i$ intersects $E_{Z}$ in a
unique point, because each curve $L_i$ intersects the curve $Z$
transversally by the point $O_i$. Then there exists a birational
morphism 
$$
\pi_{\tilde{L}}\colon\hat{\mathcal{X}}'\to\mathcal{X}
$$
that is in fact the blow-up of the union of disjoint smooth curves
$\tilde{L}_1\cup\ldots\cup\tilde{L}_r$. In particular, the
threefold $\hat{\mathcal{X}}'$ is projective.

Denote by 
$$
\hat{C}_1,\ldots,\hat{C}_r\subset \hat{\mathcal{X}}'
$$  
the
$\pi_{\tilde{L}}$-proper transform of the fibers of the morphism
$\pi_Z\vert_{E_{Z}}\colon E_{Z}\to Z$ over the points
$O_1,\ldots,O_r$ in $Z$, respectively. Then there exists
a commutative diagram,
\begin{equation}
\label{equation:main-diagram} \xymatrix{
&&\mathcal{X}'\ar@{->}[dddl]_{\pi_{Z'}}\ar@{-->}[rr]^{f}\ar@{->}[dr]_{c_C}&&
\hat{\mathcal{X}}'\ar@{->}[dl]^{c_{\hat C}}\ar@{->}[dd]^{\pi_{\tilde L}}\ar@/^5pc/@{->}[ddddd]^{p}\\%
&&&\overline{\mathcal{X}}\ar@/_1pc/@{->}[ddr]_{q}&&&&\\%
&&&&\mathcal{X}\ar@{->}[d]^{\pi_Z}&&&\\%
&S'\times\mathbb{P}^1
\ar@{->}[dr]^{p_{S'}}\ar@{->}[ddr]_{p_{\mathbb{P}^1}'}\ar@{->}[rrr]^{\pi_L}&&&
S\times\mathbb{P}^1\ar@{->}[dl]^{p_{S}}\ar@{->}[dd]^{p_{\mathbb{P}^1}}&&&\\%
&&S'\ar@{->}[r]^{\pi_O}&S&&&&\\
&&\mathbb{P}^1\ar@{=}[rr]&&\mathbb{P}^1&&&}
\end{equation} %
such that 
$q$ is the blow-up of the (singular curve)
$Z+L_1+\cdots+L_r$, $c_C$ is the contraction of the curves
$C_1,\ldots,C_r$ to the $r$ singular points (ordinary double
points) of the threefold $\bar{\mathcal{X}}$, $c_{\hat C}$
contracts the curves $\hat{C}_1,\ldots,\hat{C}_r$ on the threefold
$\hat{\mathcal{X}}'$ to the same points.

Recall from \S\ref{section:slope-stability-surface} that $S$ is
equipped with an ample divisor $L_\beta$. Let $L'_\beta$ be an
ample $\mathbb{R}$-divisor on the surface $S'$ such that
\begin{equation}
\label{equation:L-prime-L}
L'_\beta\sim_{\mathbb{R}} \pi_O^*(L_\beta)-\sum_{i=1}^r \delta_i C'_i%
\end{equation}
for some real numbers $\delta_1,\ldots,\delta_r$. Then all numbers
$\delta_1,\ldots,\delta_r$ must be positive. Denote by
$\epsilon(S',Z',L'_\beta)$ the Seshadri constant of $(S',Z')$ with
respect to $L'_\beta$. Denote by $\tau(S',Z',L'_\beta)$ the
pseudoeffective threshold of $(S',Z')$ with respect to $L'_\beta$,
i.e. the number
$$
\sup\{c>0\,:\, L'_\beta-cZ' \h{\rm \ is big}\}.%
$$
Let $c$ be a positive real number.

\begin{lemma}
\label{lemma:Seshadri-constants}  
(i)
One has
$\epsilon(S,Z,L_\beta)\ge\epsilon(S',Z',L'_\beta)$ and
$\epsilon(S',Z',L'_\beta)\le \delta_i$ for every $i$.

\noindent
(ii)
If $c<\epsilon(S,Z,L_\beta)$ and $c\ge\delta_i$
for every $i$, then the divisor $L'_\beta-cZ'$ is big and, in
particular, $\tau(S',Z',L'_\beta)>\epsilon(S',Z',L'_\beta)$.
\end{lemma}

\begin{proof}
The inequality $\epsilon(S,Z,L_\beta)\ge
\epsilon(S',Z',L'_\beta)$ is obvious. The inequality
$\epsilon(S',Z',L'_\beta)\le \delta_i$ follows from
$L'_\beta. C'_i=\delta_i$ and $Z'. C'_i=1$. Suppose that
$c<\epsilon(S,Z,L_\beta)$. Then $L_\beta-cZ$ is ample. 
Since
$$
L'_\beta-cZ'\sim_{\mathbb{R}}
\pi_O^*\big(L_\beta-cZ\big)+\sum_{i=1}^r (c-\delta_i) C'_i,
$$
we see that the divisor $L'_\beta-cZ'$ is big provided that
$c\ge\delta_i$ for every $i$.
\end{proof}

Let $\mathcal{D}'$ be the proper transform of the divisor
$\mathcal{D}$ on $\mathcal{X}'$. Put
\beq
\lb{Lbetaprime}
\mathcal{L}_\beta':=(p_{S'}\circ\pi_{Z'})^\star(
L_\beta')-cE_{Z'}. 
\eeq
If $c<\epsilon(S',Z',L'_\beta)$, then
$\mathcal{L}_\beta'$ is $p'$-ample by Lemma~\ref{lemma:p-ample}.

\begin{remark}
\label{remark:DNC-blow-up} If $\mathcal{L}_\beta'$ is $p'$-ample,
then the triple $(\mathcal{X}',\mathcal{L}_\beta',\mathcal{D}')$
is the test configuration obtained via deformation to the normal
cone of $Z'$ in $S'$.
\end{remark}

\bdefin
\lb{RDef}
Denote by $R'\subset S'\times \PP^1, 
R_{\mathcal{X}'}\subset\mathcal{X}'$, and $R_{\hat{\mathcal{X}}'}
\subset \hat{\mathcal{X}}'$
the proper transforms of the 
surface $Z\times\mathbb{P}^1\subset S\times \mathbb{P}^1$
with respect to the maps $\pi_L, \pi_L\circ\pi_{Z'}$,
and $\pi_{\tilde L}\circ \pi_Z$, respectively.
\edefin

\begin{lemma}
\label{lemma:flop-slope-Seshadri-0} Suppose that
$\epsilon(S',Z',L'_\beta)<c<\epsilon(S,Z,L_\beta)$ and
$c\ge\delta_i$ for every $i$. Then $\mathcal{L}_\beta'$ is
$p'$-big. Moreover, the curves $C_1,\ldots,C_r$ are the only
curves in $\mathcal{X}'$ that are mapped by $p'$ to points and
have negative intersections with $\mathcal{L}_\beta'$.
\end{lemma}

\begin{proof}
One has
$$
\baeq
\mathcal{L}_\beta'
&\sim_{\mathbb{R}}(p_{S'}\circ\pi_{Z'})^\star(L'_\beta)-cE_{Z'}
\\
&\sim_{\mathbb{R}}(p_{S'}\circ\pi_{Z'})^\star (L'_\beta-cZ')+c(p_{S'}\circ\pi_{Z'})^\star (Z')-cE_{Z'}\\
&\sim_{\mathbb{R}}(p_{S'}\circ\pi_{Z'})^\star (L'_\beta-cZ)
+
c\pi_{Z'}^\star (R')
-cE_{Z'}
\\
&\sim_{\mathbb{R}}(p_{S'}\circ\pi_{Z'})^\star (L'_\beta-cZ)+cR_{\mathcal{X}'}.%
\eaeq
$$
Since $L'_\beta-cZ'$ is big by
Lemma~\ref{lemma:Seshadri-constants}, we see that
$(p_{S'}\circ\pi_{Z'})^\star L'_\beta-cE_{Z'}$ is
$p_{\mathbb{P}^1}\circ\pi$-big.

Let $\Gamma$ be an irreducible curve in $\mathcal{X}'$ such that
$p(\Gamma)$ is the point $[0:1]$ and
$\mathcal{L}_\beta'.\Gamma<0$. Let us show that $\Gamma$ is
one of the curves $C_1,\ldots,C_r$. If $\pi_{Z'}(\Gamma)$ is a
point, then
$$
\mathcal{L}_\beta'.\Gamma=\Big((p_{S'}\circ\pi_{Z'})^\star(L'_\beta)-cE_{Z'}\Big).\Gamma=-cE_{Z'}.\Gamma>0.%
$$
So, $\pi_{Z'}(\Gamma)$ is not a
point. Thus, if $\Gamma\subset E_{Z'}$, then
$$
0>\mathcal{L}_\beta'.\Gamma=\Big((p_{S'}\circ\pi_{Z'})^\star(L'_\beta)-cE_{Z'}\Big).\Gamma\ge L'_\beta. Z'-cE_{Z'}.\Gamma=L'_\beta. Z'+cS_0'.\Gamma>cS_0'.\Gamma,%
$$
which implies that $S_0'.\Gamma<0$. Thus, $\Gamma\subset
S_0'$. Then,
$$
\mathcal{L}_\beta'.\Gamma=\Big((p_{S'}\circ\pi_{Z'})^\star(L'_\beta)-cE_{Z'}\Big).\Gamma=(L'_\beta-cZ').\Gamma,
$$
where we used that $S_0'\cong S'$. On the other hand,
\eqref{equation:L-prime-L} gives
$$
L'_\beta-cZ'\sim_{\mathbb{R}}
\pi_O^*\Big(L_\beta-cZ\Big)+\sum_{i=1}^r (c-\delta_i) C'_i,
$$
where $L_\beta-cZ$ is ample on $S$. Since  $c\ge\delta_i$
for every $i$ by assumption, we see that the curve $\Gamma$ must
be one of the curves $C_1',\ldots,C_r'$.
\end{proof}

A sufficient condition for the $\hat{p}'$-ampleness of the divisor
$\hat{\mathcal{L}}_\beta'$ is given by

\begin{lemma}
\label{lemma:flop-slop-ample} Suppose that
$\epsilon(S',Z',L'_\beta)<c<\epsilon(S,Z,L_\beta)$ and
$c\ge\delta_i$ for every $i$. Then
$\hat{\mathcal{L}}_\beta'$ is $\hat{p}'$-ample.
\end{lemma}

\begin{proof}
Since $L_{\beta}'$ is ample, there is a constant $\gamma_0>0$
(that depend only on $L_{\beta}'$) such that
$$
L_{\beta}'.\Omega'\ge\gamma_0
$$
for every curve $\Omega'$ in $S'$. Similarly, there is a constant
$\gamma_1>0$ (that depend on $L_{\beta}'$ and $c$ alone) such that
$$
\Big(L_\beta-cZ\Big).\Omega\ge\gamma_1%
$$
for every curve $\Omega\subset S$, because
$c<\epsilon(X,Z,L_{\beta})$. Put
$$
\gamma=\min\big\{c,\gamma_0,\gamma_1,\delta_1,\ldots,\delta_r,c-\delta_1,\ldots,c-\delta_r\big\}.
$$

Let $\Gamma$ be an irreducible curve in $\hat{\mathcal{X}}'$ that
is contracted by $\hat{p}'$ to a point. To show that
$\hat{\mathcal{L}}_\beta'$ is $\hat{p}'$-ample, it is enough to
prove that $\hat{\mathcal{L}}_{\beta}.\Gamma\ge\gamma$.

Denote by $\hat{S}_0$ and $\hat{E}_Z$ the proper transforms of the
surfaces $S_0$ and $E_Z$ on the threefold $\hat{\mathcal{X}}'$,
respectively. If $\Gamma\not\subset \hat{E}_Z\cup \hat{S}_{0}$,
then
$$
\hat{\mathcal{L}}_\beta'. \Gamma\ge\gamma_0\ge
\gamma.
$$
Thus, we may assume that $\Gamma\subset \hat{E}_Z\cup
\hat{S}_{0}$. One the other hand, it follows from
\eqref{equation:main-diagram} that
$$
\hat{S}_{0}\cong S
$$
and $\hat{\mathcal{L}}_\beta'\vert_{\hat{S}_{0}}\sim_{\mathbb{R}}
L_\beta-cZ$. 
Thus, if $\Gamma\subset\hat{S}_{0}$, then
$$
\hat{\mathcal{L}}_\beta'.\Gamma\ge \gamma_1\ge
\gamma.
$$
Hence, we may assume that $\Gamma\subset\hat{E}_Z$.

Denote by $\hat{F}_1,\ldots,\hat{F}_r$ the exceptional divisors of
$\pi_{\tilde L}$. We may assume that
$\pi_{\tilde{L}}(\hat{F}_i)=\tilde L_i$ for every $i$. 
Using \eqref{Lbetaprime} and \eqref{equation:L-prime-L} gives
\begin{equation}
\label{equation:lemma:flop-slop-ample}
\hat{\mathcal{L}}_\beta'
\sim_{\mathbb{R}}
(p_{S}\circ\pi_{Z}\circ\pi_{\tilde{L}})^\star
(L_\beta)-\sum_{i=1}^{r}\delta_i\hat{F}_i-c\hat{E}_{Z}.
\end{equation}
If $\pi_{\tilde{L}}(\Gamma)$ is a point $\tilde{L}_i\cap E_Z$,
then
$$
\hat{\mathcal{L}}_\beta'.\Gamma=\delta_i\ge \gamma.
$$
If
$\Gamma=\hat{C}_i$, then
$$
\hat{\mathcal{L}}_\beta'.\Gamma=c-\delta_i\ge \gamma.
$$
If
$\Gamma$ is contracted by $\pi_{Z}\circ\pi_{\tilde{L}}$ to a point
in $Z$ that is different from $O_1,\ldots,O_r$, then
$$
\hat{\mathcal{L}}_\beta'.\Gamma=c\ge \gamma.
$$
Thus, we may assume that $\pi_{\tilde{L}}\circ\pi_{Z}(\Gamma)=Z$.
In particular, we see that 
\beq
\lb{notcontainedEq}
\h{$\Gamma$ is not contained in any
divisor $\hat{F}_i$.}
\eeq
Rewriting 
\eqref{equation:lemma:flop-slop-ample} and using the fact
that $\hat{E}_Z\cup_i \hat F_i$ is the exceptional divisor
of $\pi_{Z}\circ\pi_{\tilde{L}}$ gives (recall
Definition \ref{RDef}),
$$
\baeq
\hat{\mathcal{L}}_\beta'
&\sim_{\mathbb{R}}
(p_{S}\circ\pi_{Z}\circ\pi_{\tilde{L}})^\star
(L_\beta-cZ)+c(p_{S}\circ\pi_{Z}\circ\pi_{\tilde{L}})^\star
Z-\sum_{i=1}^{r}\delta_i\hat{F}_i-c\hat{E}_{Z}
\\
&\sim_{\mathbb{R}} 
(p_{S}\circ\pi_{Z}\circ\pi_{\tilde{L}})^\star
(L_\beta-cZ)+c(\pi_{Z}\circ\pi_{\tilde{L}})^\star
(Z\times\PP^1)-\sum_{i=1}^{r}\delta_i\hat{F}_i-c\hat{E}_{Z}
\\
&\sim_{\mathbb{R}} 
(p_{S}\circ\pi_{Z}\circ\pi_{\tilde{L}})^\star(L_\beta-cZ)
+cR_{\hat{\mathcal{X}}'}+\sum_{i=1}^{r}(c-\delta_i)\hat{F}_i.
\eaeq
$$
Thus, if $\Gamma\not\subset R_{\hat{\mathcal{X}}'}$, then 
since $\Gamma$ is a finite cover of $Z$, degree consideration give
$$
\baeq
\hat{\mathcal{L}}_\beta'.\Gamma
&=(p_{S}\circ\pi_{Z}\circ\pi_{\tilde{L}})^\star
(L_\beta-cZ).\Gamma+cR_{\hat{\mathcal{X}}'}.\Gamma+\sum_{i=1}^{r}(c-\delta_i)\hat{F}_i.\Gamma
\\
&\ge (L_\beta-cZ).
Z+cR_{\hat{\mathcal{X}}'}.\Gamma+\sum_{i=1}^{r}(c-\delta_i)\hat{F}_i.\Gamma
\cr
&\ge (L_\beta-cZ).
Z\ge \gamma_1\ge \gamma,
\eaeq
$$
where we also used \eqref{notcontainedEq}.
Thus, we may assume that $\Gamma\subset R_{\hat{\mathcal{X}}'}$. Then
$\Gamma$ is the proper transform of the curve $E_Z\cap
R_{\mathcal{X}}$. Since the surfaces $S_0$ and $R_{\mathcal{X}}$
are disjoint, we have $S_0.\pi_{\tilde{L}}(\Gamma)=0$. Then
$$
\baeq
\hat{\mathcal{L}}_\beta'.\Gamma
&=
(p_{S}\circ\pi_{\tilde{L}}\circ\pi_{Z})^\star(L_\beta).\Gamma
-
\sum_{i=1}^{r}\delta_i\hat{F}_i.\Gamma
-
c\hat{E}_{Z}.\Gamma
\cr
&=L_\beta.
Z-\sum_{i=1}^{r}\delta_i\hat{F}_i.\Gamma-c\hat{E}_{Z}.\Gamma
\cr
&=L_\beta.
Z-\sum_{i=1}^{r}\delta_i-c\hat{E}_{Z}.\Gamma\\
&=L_\beta.
Z-\sum_{i=1}^{r}\delta_i-cE_{Z}.\pi_{\tilde{L}}(\Gamma)
\cr
&=L_\beta.
Z-\sum_{i=1}^{r}\delta_i+cS_{0}.\pi_{\tilde{L}}(\Gamma)\\
&=L_\beta. Z-\sum_{i=1}^{r}\delta_i=L'_\beta. Z'\ge
\gamma_0\ge\gamma.
\eaeq
$$
This completes the proof of the lemma.
\end{proof}

Let $\hat{\mathcal{D}}'$ be the proper transform of the divisor
$\mathcal{D}$ on the threefold $\hat{\mathcal{X}}'$, and let
$\hat{\mathcal{L}}_\beta'$ be the proper transform of (the class in
$\mathrm{Pic}(\hat{\mathcal{X}}')\otimes\mathbb{R}$ of) the
divisor $\mathcal L_\beta'$ on the threefold $\hat{\mathcal{X}}'$ (note that
$\hat{\mathcal{L}}_\beta'$ is well-defined, since $f$ is an
isomorphism in codimension one).

\begin{cor}
\label{cor:floped-DNC-blow-up} 
Suppose that
$\epsilon(S',Z',L'_\beta)<c<\epsilon(S,Z,L_\beta)$ and
$c\ge\delta_i$ for every $i$.
Then the quadruple
$(\hat{\mathcal{X}}',\hat{\mathcal{L}}_\beta',\hat{\mathcal{D}}',\be)$
is a test configuration.
\end{cor}

Next, we compute the generalized Futaki invariant of the flopped test
configuration.

\begin{proposition}
\label{proposition:Futaki-flop-slop} 
Suppose that
$$
L_\beta'\sim_{\mathbb{R}}-K_{S'}-(1-\beta)C'.
$$
Then (recall (\ref{equation:intersection-formula})), 
$$
\baeq
\mathrm{F}(\hat{\mathcal{X}}',\hat{\mathcal{L}}_\beta',\hat{\mathcal{D}}',\be)
&=
\mathrm{F}(\mathcal{X}', \mathcal{L}_\beta',\mathcal{D}',\be)
-
2\sum_{i=1}^{r}(\mathcal{L}_\beta'.C_i)^3
-
3(1-\beta)\sum_{i=1}^{r}(\mathcal{L}_\beta'.C_i)^2(\mathcal{D}'.C_i)
\\
&
=
2(\mathcal{L}_\beta')^3
+
3
\Big(
K_{\mathcal{X}'}
-
(p')^\star\big(K_{\mathbb{P}^1}\big)
+
(1-\beta)\mathcal{D}'\Big).
(\mathcal{L}_\beta')^2
\cr
&
\q 
-
2\sum_{i=1}^{r}(\mathcal{L}_\beta'.C_i)^3
-
3(1-\beta)\sum_{i=1}^{r}(\mathcal{L}_\beta'.C_i)^2(\mathcal{D}'.C_i).
\eaeq
$$
\end{proposition}

\begin{proof}
Recall from \S\ref{section:DF} that
$$
\mathrm{F}(\hat{\mathcal{X}}',\hat{\mathcal{L}}_\beta',\hat{\mathcal{D}}',\be)=2(\hat{\mathcal{L}}_\beta')^3
+3
\Big(
K_{\hat{\mathcal{X}}'}
-
(\hat{p}')^\star\big(K_{\mathbb{P}^1}\big)
+(1-\beta)\hat{\mathcal{D}}'
\Big).
(\hat{\mathcal{L}}_\beta')^2.
$$
The assertion now follows from
\eqref{equation:intersection-formula} and 
Lemma \ref{lemma:flops-intersections}, together 
with the fact that, as in \eqref{intersectionKVCEq}, 
$K_{{\mathcal X}'}.C_i'=0$, while $p^\star K_{\PP^1}.C_i'=0$
since the $C_i$ are contained in the central fiber of $p$.
\end{proof}

\section{Proof of Theorem \ref{MainThm}}
\label{section:proof}

According to \cite[Theorem 1.4]{CR}, all ALF surfaces
$(S,C)$ such that $-K_S-C$ is big satify either
(i) $-K_S-C$ is ample, or (ii) $S$ is obtained
from an ALF surface $(s,c)$ such that $-K_s-c$ is ample
by blowing-up $s$ at $r>0$ distinct points on $c$
and letting $C$ denote the proper transform of $c$.
Proposition \ref{theorem:Maeda} already established
Theorem \ref{MainThm} in the case (i) holds. To 
complete the proof of Theorem \ref{MainThm} it remains
to handle case (ii).

To that end, we switch back 
to the notation and assumptions of \S\ref{section:flops}.
We suppose that $(S,C)$ is such that  $-K_{S}-C$ is ample
(hence ALF), and that $(S',C')$ is still ALF, i.e., 
$$
L_\beta':=-K_{S'}-(1-\beta)C'
$$ 
is ample for all sufficiently small $\beta$. 
Note that
$-K_{S'}-C'=\pi_O^\star(-K_S-C)$ is big being the pull-back
under a birational map of an ample class.
Thus, $(S',C')$ satisfies the assumptions of Theorem \ref{MainThm}.
However, it is not possible to slope destabilize this latter pair
in the same way as was done for $(S,C)$ in \S\ref{section:maeda-class}. Indeed, 
\beq
\lb{LbetaprEq}
L_\beta'
\sim_{\mathbb{R}} \pi_O^\star(-K_{S}-(1-\beta)C)-\beta\sum_{i=1}^rC_i',%
\eeq
so 
that by Lemma \ref{lemma:Seshadri-constants}   (i) 
(putting $\delta_i=\beta$ and $Z=C$),
$\epsilon(S',Z',L'_\beta)\le\beta$, 
and in particular using Proposition \ref{proposition:formula} one checks
that the generalized Futaki invariant 
$F(\mathcal{X}',\mathcal{L}_{\be,c}',\mathcal{D},\be)$ 
of the degeneration
to the normal cone 
is positive for $c\in(0,\be)$, and so $(S',C')$ is not slope destabilized in this way.
In what follows, we apply the results of \S\ref{section:flops} to destabilize our
pair nevertheless.

Before proving Theorem~\ref{MainThm}, let us consider
a model example.

\begin{example}
\label{example:blow-up-conic}
{\rm
Suppose that $S=\mathbb{P}^2$ and
$C$ is a smooth conic. Then $\epsilon(S',L_\beta',Z')=\beta$.
Thus, if $c<\beta$, then $\mathcal{L}_{\beta,c}'$ is $p'$-ample. By
Proposition~\ref{proposition:formula}, we have
$$
\mathrm{F}(\mathcal{X}', \mathcal{L}_{\beta,c}',\mathcal{D}',\be)
=\big(6\beta c-3c^2\big)\big(2+\beta(4-r)\big)+\big(2c^3-3c^2\beta\big)\big(4-r\big).%
$$
In particular, this invariant is always positive for $\be$ sufficiently small
(depending on $r$).
On the other hand,
$$
\tau(S',L_\beta',Z')=\epsilon(S,L_\beta,Z)=\frac{1}{2}+\beta.
$$
Thus, if $\beta<c<\frac{1}{2}+\beta$, then
$\hat{\mathcal{L}}_\beta'$ is $\hat{p}'$-ample by
Lemma~\ref{lemma:flop-slop-ample}. By
Proposition~\ref{proposition:Futaki-flop-slop}, one has
\begin{multline*}
\mathrm{F}(\hat{\mathcal{X}}',\hat{\mathcal{L}}_{\beta,c}',\hat{\mathcal{D}}',\be)=\mathrm{F}(\mathcal{X}',
\mathcal{L}_{\beta,c}',\mathcal{D}',\be)+2r(c-\beta)^3=\\
=\big(6\beta
c-3c^2\big)\big(2+\beta(4-r)\big)+\big(2c^3-3c^2\beta\big)\big(4-r\big)+2r(c-\beta)^3
\end{multline*}
(see Appendix~\ref{appendix}). If we put $c=\frac{1}{2}+\beta$,
then
$$
\lim_{\beta\to
0^+}\mathrm{F}(\hat{\mathcal{X}}',\hat{\mathcal{L}}_\beta',\hat{\mathcal{D}}',\be)=-\frac{1}{2}.
$$
}
\end{example}

Recalling the discussion at the beginning of this section, 
Theorem~\ref{MainThm} follows from
the following result.

\begin{proposition}
\label{proposition:big-nef} 
The triple $(S',C',\be)$ is flop-slope unstable
for all sufficiently small $\beta$.
\end{proposition}

\begin{proof}
Let $\epsilon(S,Z,-K_{S}-C)$ be the Seshadri constant of
$Z\subset S$ with respect to $-K_{S}-C$. Pick any positive
$\gamma<\epsilon(S,Z,-K_{S}-C)$. Then
$$
-K_{S}-C\sim_{\mathbb{R}}\gamma Z+H
$$
for some ample $\mathbb{R}$-divisor $H$. Then
$$
L_\beta\sim_{\mathbb{R}}-K_{S}-C+\beta
C\sim_{\mathbb{R}}(\gamma+\beta)C+H,
$$
hence
$\epsilon(S,L_\beta,Z)\ge\gamma+\beta$, 
so
$
\epsilon(S,L_\beta,Z)>\beta\ge\epsilon(S',L'_\beta,Z')
$
by Lemma \ref{lemma:Seshadri-constants} (i).
By taking $\be$ small, we may suppose that $\gamma>\beta$. 
Letting $c$ be a real number such
that 
\beq
\lb{SeshadriIneqsEq}
\epsilon(S',Z',L'_\beta)\le\beta< c \le \gamma< \epsilon(S,L_\beta,Z),
\eeq
Lemma~\ref{lemma:flop-slop-ample} implies that $\hat{\mathcal{L}}_{\beta,c}'$ is
$\hat{p}'$-ample. 
By
Proposition~\ref{proposition:Futaki-flop-slop}, we have
$$
\mathrm{F}(\hat{\mathcal{X}}',\hat{\mathcal{L}}_{\beta,c}',\hat{\mathcal{D}}',\be)=\mathrm{F}(\mathcal{X}', \mathcal{L}_{\beta,c}',\mathcal{D}',\be)
-2\sum_{i=1}^{r}(\mathcal{L}_{\beta,c}'.C_i)^3-3\sum_{i=1}^{r}(\mathcal{L}_{\beta,c}'.C_i)^2(\mathcal{D}'.C_i)\\
$$
Moreover, by Proposition~\ref{proposition:formula}
$$
\mathrm{F}(\mathcal{X}', \mathcal{L}_{\beta,c}',\mathcal{D}',\be)=(6\beta c-3c^2)L_\beta'.C'+(2c^3-3c^2\beta)(C')^2.%
$$
Note that using \eqref{Lbetaprime} and \eqref{LbetaprEq},
$$
\mathcal{L}_{\beta,c}'.C_i
=
\big(
(p_{S'}\circ\pi_{Z'})^\star L_\be' - cE_{Z'}
\big)
.C_i
=
L_\be'.C_i'-cZ'.C_i'=\be-c.
$$
In addition, before the blow-up $\pi_{Z'}$, the intersection of
$\mathcal D=C'\times \PP^1$ and $S'\subset S'\times \PP^1$ is precisely 
$Z'\subset S'\subset S'\times \PP^1$ (this is precisely where
we use that $Z'=C'$). Thus, after blowing-up $Z'$,
the surfaces $\mathcal D'$ and $S'_0\cong S'$ (recall \eqref{SprzeroEq}) no longer intersect.
Since $C_i$ is contained in $S'_0$,
$$
\mathcal{D}'.C_i=0.
$$
Combining these facts,
$$
\mathrm{F}(\hat{\mathcal{X}}',\hat{\mathcal{L}}_{\beta,c}',\hat{\mathcal{D}}',\be)=(6\beta c-3c^2)L_\beta'.C'+(2c^3-3c^2\beta)(C')^2+2r(c-\beta)^3.%
$$
By Remark \ref{RationalRemark}, $C$ and hence also $C'$ are rational, so
$L_\be'.C'=-(K_{S'}-(1-\be)C').C'=2+\be C'^2$. Thus, 
putting $c=\gamma$ and grouping most terms of order $\be$ together yields, 
\beq
\label{FutakiLongEq}
\baeq
\mathrm{F}(\hat{\mathcal{X}}',\hat{\mathcal{L}}_{\beta,\gamma}',\hat{\mathcal{D}}',\be)
&= 
-\gamma^2L_\beta'.C'
-2\gamma^2\big(L_\beta'-\gamma
C').C'+2r\gamma^3
\cr
&\qq
+\beta\big(6\gamma
L_\beta'.C'-3\gamma^2C'^2-6r\gamma^2+6r\beta\gamma-2r\beta^2\big)\cr
&= 
-
\gamma^2(2+\beta C'^2)
-
2\gamma^2\big(\pi_O^\star(L_\beta-\gamma C)+(\gamma-\beta)\sum_{i=1}^rC_i'\big).C'
+
2r\gamma^3 
\cr&\qq
+\beta\big(6\gamma
L_\beta'.C'-3\gamma^2C'^2-6r\gamma^2+6r\beta\gamma-2r\beta^2\big)\cr
& =-\gamma^2(2+\beta C'^2)-2\gamma^2\big(L_\beta-\gamma
C\big).C-2\gamma^2(\gamma-\beta)r+2r\gamma^3 \cr
&\qq+\beta\big(6\gamma
L_\beta'.C'-3\gamma^2C'^2-6r\gamma^2+6r\beta\gamma-2r\beta^2\big)
\cr
&=
-2\gamma^2-2\gamma^2\big(L_\beta-\gamma
C\big).C
+
\beta
\big(
6\gamma L_\beta'.C'-4\gamma^2C'^2-4r\gamma^2+6r\beta\gamma-2r\beta^2
\big),
\cr
\eaeq 
\eeq 
so by \eqref{SeshadriIneqsEq}, 
\beq 
\baeq\label{RestrictEq}
\mathrm{F}(\hat{\mathcal{X}}',\hat{\mathcal{L}}_{\beta,c}',\hat{\mathcal{D}}',\be)
<-2\gamma^2+\beta\big(6\gamma
L_\beta'.C'-4\gamma^2C'^2-4r\gamma^2+6r\beta\gamma-2r\beta^2\big).
\eaeq \eeq implying that 
$
\lim_{\beta\to0^+}
\mathrm{F}(\hat{\mathcal{X}}',\hat{\mathcal{L}}_{\beta,c}',\hat{\mathcal{D}}',\be)
\le -2\gamma^2<0$.
\end{proof}

\section{Further examples}
\label{ExampSec}

We close by illustrating the advantage of using
flop-slope stability over slope stability with two simple examples.

\subsection{$\FF_1$}
According to Ross--Thomas \cite[Examples 5.27,5.35]{RT1}
(cf. Panov--Ross \cite[Example 3.8]{PanovRoss}), $\FF_1$ is (for $\be=1$) slope destabilized by
the $-1$-curve. More generally, by Li--Sun \cite{LS}, the Futaki invariant
of the slope test configuration of the triple $(\FF_1,C,\be)$ with
$C$ smooth in $|-K_{\FF_1}|$ and with respect to the $-1$-curve
equals $-3c^2\beta-2c^3+3c^2+6c\beta$, which for $c=2\beta$ (the
Seshadri constant in this case), gives $4\beta^2(6-7\beta)$,
showing that there exists no KEE metric when $\be\in(6/7,1]$.
However, $\FF_1$ is {\it not}
destabilized by any fiber of its natural projection to $\PP^1$
\cite[Theorem 1.3]{PanovRoss}. We
now show that $\FF_1$ {\it is} destabilized by a fiber after one
flop, and this even holds for $\be\in(12/13,1]$. 

To show this, it
is most convenient to carry over the notation and assumptions of
\S\ref{section:flops}. Thus, we let $S$ be $\PP^2$, $C$ be a smooth
cubic, and $Z$ be a line. Then $S'=\FF_1$ is the blow-up of $S$ at a
point $O_1\in Z\cap C$,  $C'$ is an elliptic (anticanonical)
curve, and $Z'$ is a fiber of the natural projection $\FF_1\ra
\PP^1$. In addition $\mathcal D$ is $C'\times\PP^1$
and $\mathcal D'$ is its proper transform with respect to
the blow-up of $Z'\subset S'\times\PP^1$. 
Let $L_\beta':=-K_{S'}-(1-\beta)C'=\beta C'$, so
$\epsilon(S',L_\beta',Z')=\beta$. 
As $L_\be'.Z'=2\be$ and $Z'^2=0$, Proposition~\ref{proposition:formula} gives
\beq
\lb{slopeFutakiF1Eq}
\mathrm{F}(\mathcal{X}',\mathcal{L}_{\beta,c}',\mathcal{D}',\be)=6c\beta(2-c).
\eeq
Thus, if $c<\beta$, then
$\mathrm{F}({\mathcal{X}}',{\mathcal{L}}_{\beta,c}',{\mathcal{D}}')>0$.
On the other hand, we have
$$
\tau(S',L_\beta',Z')=\epsilon(S,L_\beta,Z)=3\beta.
$$
Thus, it follows from Lemma~\ref{lemma:flop-slop-ample} that
$\hat{\mathcal{L}}_\beta$ is ample for every $c\in (\beta,3\beta)$. By
Proposition~\ref{proposition:Futaki-flop-slop} and \eqref{slopeFutakiF1Eq}, 
\beq\baeq
\mathrm{F}(\hat{\mathcal{X}}',\hat{\mathcal{L}}_{\beta,c}',\hat{\mathcal{D}}',\be)
&=
\mathrm{F}(\mathcal{X}',\mathcal{L}_{\beta,c}',\mathcal{D}',\be)
-2(\beta-c)^3
-3(1-\beta)(\beta-c)^2
\cr 
&=6c\beta(2-c)-2(\beta-c)^3-3(1-\beta)(\beta-c)^2. \eaeq \eeq
If $c=3\beta$, then
$\mathrm{F}(\hat{\mathcal{X}}',\hat{\mathcal{L}}_{\beta,c}',\hat{\mathcal{D}}')=24\beta^2-26\beta^3$,
which implies that
$\mathrm{F}(\hat{\mathcal{X}}',\hat{\mathcal{L}}_{\beta,c}',\hat{\mathcal{D}}',\be)<0$
(for some $c\in(\beta,3\beta)$) provided that
$\beta>\frac{12}{13}$. 

In fact, one can show that $(\FF_1,C',\be)$ does not admit 
a KEE metric for  $\beta\in(\frac{4}{5},1]$
\cite{Gabor2012}. On the other hand, $(\FF_1,C',\be)$ admits 
a KEE metric for $\beta\in(0,\frac{3}{10})$, and, moreover,
if $C'$ is a \emph{general} curve in $|-K_{\FF_1}|$, then $(\FF_1,C',\be)$
admits a KEE metric for $\beta\in(0,\frac{3}{7})$ \cite[Corollary~1.16]{CJ}.

\subsection{$\h{\rm Bl}_{O_1,O_2}\PP^2$}

We take, as in the previous subsection, $S=\PP^2, C$ a smooth
cubic, and $Z$ a line, but now blow-up two points $O_1,O_2\in
C\cap Z$ to obtain $S'$, and let $C',Z'$ be the proper transforms
of $C,Z$, respectively. According to Panov--Ross \cite[Example 7.6]{PanovRoss}, 
the surface $S'$ (with
$\be=1$) is slope stable. We will show that it is {\it not } flop-slope
stable, and moreover this holds also for
$(S',C',\be)$ with  $\be\in(21/25,1]$. By
comparison, Sz\'ekelyhidi \cite{Gabor2012} constructed a
destabilizing toric degeneration for $\beta\in(\frac{7}{9},1]$ in
the case when $C'$ does not contain either of the points $Z'\cap
C_1'$ or $Z'\cap C_2'$, where $C_i'$ are the exceptional curves of
the blow-down map to $\PP^2$. It is interesting to note that the
value $21/25$ also arises in the related smooth continuity method
\cite[Proposition 10]{Szek2011},\cite[Example 2]{Li2012}.

By Proposition~\ref{proposition:formula}, we have
$$
\mathrm{F}(\mathcal{X}', \mathcal{L}_{\beta,c}',\mathcal{D}',\be)=3\beta c(2-c)-c^2(2c-3).%
$$
Here $c<\epsilon(S',L_\be',Z')=\beta$.
 Thus, $\mathrm{F}(\mathcal{X}',
\mathcal{L}_\beta',\mathcal{D}',\be)>0$ for every $c\in(0,\beta)$
(i.e., slope stable). On the
other hand, we have
$$
\tau(S',L_\beta',Z')=\epsilon(S,L_\beta,Z)=3\beta.
$$
By
Lemma~\ref{lemma:flop-slop-ample}, the divisor
$\hat{\mathcal{L}'}_\beta$ is ample for
$c\in (\be,3\be)$. 
Note that $C_i'.Z'=1$ and as in \eqref{intersectionKVCEq}
(see also \eqref{AppintersectionEq})
$K_{\mathcal{X}'}.C_i'=0$. Therefore,
$$
\mathcal{L}_\beta'.C_i=-\beta
K_{S'}.C_i'-cZ'.C_i'=\beta-cZ'.C_i'=\beta-c,
$$
and by
Proposition~\ref{proposition:Futaki-flop-slop}, one has
$$
\mathrm{F}(\hat{\mathcal{X}}',\hat{\mathcal{L}}_\beta',\hat{\mathcal{D}}',\be)=\mathrm{F}(\mathcal{X}',
\mathcal{L}_\beta',\mathcal{D}',\be)-4(\beta-c)^3-6(\beta-c)^2(1-\beta).
$$
Plugging-in $c=3\beta$ yields
$$
\mathrm{F}(\hat{\mathcal{X}}',\hat{\mathcal{L}}_\beta',\hat{\mathcal{D}}',\be)
=9\beta^2(2-3\beta)-9\beta^2(6\beta-3)+
32\beta^3-24\beta^2(1-\beta)=\beta^2(21-25\beta)< 0,
$$
when $\beta> \frac{21}{25}$. 

Note that $(S',C',\be)$ admits a KEE metric  for  $\beta\in(0,\frac{3}{7})$,
and, moreover, if $C'$ does not
contain neither of the points $Z\cap C_1'$ and $Z\cap C_2'$,
then a KEE metric exists
for  $\beta\in(0,\frac{1}{2})$ \cite[Corollary~1.16]{CJ}.

\appendix
\section{Simple flops}
\label{appendix}

Let $V$ be a smooth projective variety, and let 
$$
C\subset V
$$ 
be a smooth
rational curve such that its normal bundle in $V$ is isomorphic to
$\mathcal{O}_{\mathbb{P}^1}(-1)\oplus
\mathcal{O}_{\mathbb{P}^1}(-1)$. Then there exists a commutative
diagram
$$
\xymatrix{
&\mskip80mu\mathrm{Bl}_CV=W=\mathrm{Bl}_{\hat C}\hat V\mskip80mu\ar@{->}[dl]_{\pi_C}\ar@{->}[dr]^{\pi_{\hat C}}&\\%
V\ar@{->}[dr]_{c_C}\ar@{-->}[rr]^{f}&&\hat V\ar@{->}[dl]^{c_{\hat C}}\\
&X&}
$$ %
such that the threefold $\hat V$ is smooth, the threefold $X$ has
an isolated ordinary double point, $\pi_C$ is a blow-up of the
curve $C$, $\pi_{\hat C}$ is the contraction of the
$\pi_C$-exceptional surface, let us call it
$E\cong\mathbb{P}^1\times\mathbb{P}^1$, to a smooth rational
curve, let us call it $\hat C$. We define the map $f$ by declaring
the diagram to be commutative. This defines $f$ as a birational
map away from $C$. It is important in this construction that
$\pi_C\not=\pi_{\hat C}$, so that the map $f$ is not an
isomorphism. Finally, $c_C$ and $c_{\hat C}$ are (small)
contractions of the curves $C$ and $\hat C$, respectively, to the
isolated ordinary double point of $X$.

\begin{remark}
\label{remark:Kulikov} The birational map $f\colon
V\dasharrow\hat{V}$ is called the simple flop of the curve $C$.
Sometimes it is called an Atiyah flop \cite{Atiyah}. 
Later it was explicitly
introduced by Kulikov in \cite[\S4.2]{Kulikov} as
\emph{perestroika~I}.
\end{remark}

Note that the normal bundle of $\hat C$ in $U$ is isomorphic to
$\mathcal{O}_{\mathbb{P}^1}(-1)\oplus
\mathcal{O}_{\mathbb{P}^1}(-1)$. As in \eqref{intersectionKVCEq}, 
\beq
\lb{AppintersectionEq}
K_{V}.C=0.
\eeq
In fact, another way to see this equality is by 
noting that
since $c_C$ is an isomorphism away from codimension $2$, then
$K_V\sim_\mathbb{Q} c_C^\star K_V$, and of course
$c_C^\star(K_X).C=0$ since $c_C$ contracts $C$.
Similarly, $K_{\hat
V}.\hat C=0$ by construction. Moreover, we have $E\vert_E$ is a
divisor on $E\cong\mathbb{P}^1\times\mathbb{P}^1$ of bi-degree
$(-1,-1)$. Furthermore, the morphism $c_C\circ \pi_C=c_{\hat
C}\circ \pi_{\hat C}$ is just the contraction of the surface $E$
to the the isolated ordinary double point of $X$, i.e. its inverse
map is the blow-up of this point.

\begin{remark}
\label{remark:projectivity} 
Note that in general $\hat V$ is not necessarily
projective. However, it is not hard to see that $\hat V$ is
projective in many cases, either by explicit construction or
by using log MMP. In all our applications, $\hat V$ is projective
by construction, see \S\ref{section:flops}.
\end{remark}

Given an irreducible reduced Weyl divisor $D$ on $V$, we denote by
$\hat D$ the unique divisor on $\hat V$ such that 
$$
\hat
D:=
\overline{f(D\setminus C)}.
$$ 
By linearity, we extend the same
notation to all $\mathbb{R}$-divisors on $V$.
The following formula may be known, but we provide
a proof since we were not able to find a reference for it.

\begin{lemma}
\label{lemma:flops-intersections} Let $H_i, i=1,2,3,$ be
$\mathbb{R}$-divisors on $V$.  Then,
$$
\hat H_1.\hat H_2.\hat H_3=H_1.H_2.H_3-(H_1.C)(H_2.C)(H_3.C).
$$
\end{lemma}

\begin{proof}
Let ${\tilde H}_1$, ${\tilde H}_2$ and ${\tilde H}_3$ be the proper transforms of the
divisors $H_1$, $H_2$ and $H_3$ on $W$, respectively. Recall
that
$E=\mathbb{P}^1\times\mathbb{P}^1$ denotes the exceptional divisor
of $\pi_C$ (and of $\pi_{\hat C}$). 
Then,
$$
\left\{\aligned%
&{\tilde H}_1\sim_{\mathbb{R}} c_C^{\star}H_1-m_1E\sim_{\mathbb{R}} c_{\hat C}^{\star}\hat H_1-\hat m_1E,\\
&{\tilde H}_2\sim_{\mathbb{R}} c_C^{\star}H_2-m_2E\sim_{\mathbb{R}} c_{\hat C}^{\star}\hat H_2-\hat m_2E,\\
&{\tilde H}_3\sim_{\mathbb{R}} c_C^{\star}H_3-m_3E\sim_{\mathbb{R}} c_{\hat C}^{\star}\hat H_3-\hat m_3E,\\
\endaligned
\right.
$$
for some real numbers $m_i, \hat m_i$. 
Put 
$$
r_i:=H_i.C, \q 
\hat r_i:=\hat H_i.\hat C.
$$
Then each ${\tilde H}_i|_{E}$ is a divisor (in
$\mathbb{P}^1\times\mathbb{P}^1$) of bi-degree
$$
(r_i+m_i, m_i)=(\hat m_i,\hat r_i+\hat m_i);
$$
this is because $E|_E=N_{E|V}$ is a line bundle of bi-degree
$(-1,-1)$, while since $c_C(E)=C$ and $c_{\hat C}(E)=\hat C$,
\beq
\lb{firstfiberEq}
c_C^{\star}H_i|_E=H_i.c_C(E)\times\hbox{(fiber of projection of
$\pi_C$)}
=r_i\times\hbox{(bi-degree (1,0) curve)},
\eeq
and
$$c_{\hat C}^{\star}\hat H_i|_E
=
\hat H_i.c_{\hat C}(E)\times\hbox{(fiber of projection of
$\pi_{\hat C}$)}
=\hat r_i\times\hbox{(bi-degree (0,1) curve)}.
$$
Thus,
\begin{equation}\label{equation:MiNiEq} \hat m_i=r_i+m_i,
\ \hbox{ and\ } \hat r_i=-r_i. 
\end{equation} 
Now, $E^3=E|_E.E|_E=
c_1(\mathcal{O}_{\PP^1}(-1)\oplus \mathcal{O}_{\PP^1}(-1))^2=2$, 
and by \eqref{firstfiberEq},
$$
c_C^{\star}H_i.E^2=c_C^{\star}H_i|_E.E=-H_i|_E.C=-r_i.
$$
On the other hand,
$$
c_C^{\star}H_i.c_C^{\star}H_j.E=
c_C^{\star}H_i|_E.c_C^{\star}H_j|_E=0
$$
since by \eqref{firstfiberEq}, 
$c_C^{\star}H_i|_E$ and $c_C^{\star}H_j|_E$ are fibers of the same projection in $\PP^1\times \PP^1=E$.
Altogether,
$$
\begin{aligned}
{\tilde H}_1.{\tilde H}_2.{\tilde H}_3
&=
\big(c_C^{\star}H_1-m_1E\big)
.\big(c_C^{\star}H_2-m_2E\big)
.\big(c_C^{\star}H_3-m_3E\big)\\
&=H_1.H_2.H_3+\big(m_1m_2c_C^{\star}H_3+m_1m_3c_C^{\star}H_2
+m_2m_3c_C^{\star}H_1\big).E^2-m_1m_2m_3E^3
\\
&=H_1.H_2.H_3-\big(m_1m_2r_3+m_1m_3r_2+m_2m_3r_1\big)-2m_1m_2m_3.
\end{aligned}
$$
Similarly,
$$
\begin{aligned} {\tilde H}_1.{\tilde H}_2.{\tilde H}_3
&=\big(c_{\hat C}^{\star}\hat H_1-\hat m_1E\big).
\big(
c_{\hat C}^{\star}\hat H_2-\hat m_2E
\big).
\big(
c_{\hat C}^{\star}\hat H_3-\hat m_3E
\big)
\\
&=\hat H_1.\hat H_2.\hat H_3+\big(\hat m_1\hat m_2c_{\hat C}^{\star}\hat H_3+\hat m_1\hat m_3c_{\hat C}^{\star}\hat H_2+\hat m_2\hat m_3c_{\hat C}^{\star}\hat H_1\big).E^2-\hat m_1\hat m_2\hat m_3E^3\\
&=\hat H_1.\hat H_2.\hat H_3
-
\big(
\hat m_1\hat m_2\hat r_3
+
\hat m_1\hat m_3\hat r_2
+
\hat m_2\hat m_3\hat r_1
\big)-2\hat m_1\hat m_2\hat m_3.
\end{aligned}
$$
Thus,
$$
\begin{aligned}
& H_1.H_2.H_3-m_1m_2r_1-m_1m_3r_2-m_2m_3r_1-2m_1m_2m_3 \cr &\qquad
=\hat H_1.\hat H_2.\hat H_3-\hat m_1\hat m_2\hat r_3-\hat m_1\hat
m_3\hat r_2-\hat m_2\hat m_3\hat r_1-2\hat m_1\hat m_2\hat m_3.
\end{aligned}
$$
By \eqref{equation:MiNiEq}, this yields $\hat H_1.\hat H_2.\hat
H_3=H_1.H_2.H_3-r_1r_2r_3$.
\end{proof}

\end{document}